\documentclass[a4paper,12pt,]{article}
\usepackage{amsmath}
\usepackage{amsfonts}
\usepackage{amsthm,setspace}
\usepackage{graphicx}
\usepackage[usenames,dvipsnames]{pstricks}
\usepackage{epsfig}
\usepackage{pst-grad} 
\usepackage{pst-plot} 
\newtheorem{theorem}{Theorem}[section]

\newtheorem{thm}[theorem]{Theorem}

\newtheorem{cor}[theorem]{Corollary}
\theoremstyle{definition}

\newtheorem{dfn}[theorem]{Definition}

\def\ni{\noindent}

\author{Jessica Pereira\footnote{\textbf{E-mail:}jessica@goa.bits-pilani.ac.in}\ \  \ and Tarkeshwar Singh\footnote{\textbf{E-mail:}tksingh@goa.bits-pilani.ac.in}\\
Department of Mathematics,\\ Birla Institute of Technology and
Science Pilani\\ K K Birla, Goa Campus, NH-17B, Zuarinagar,\\
Goa, India \\ \\ S. Arumugam\footnote{{\bf E-mail:}s.arumugam.klu@gmail.com}\\National Centre for Advanced Research in Discrete Mathematics \\ Kalasalingam University\\
Anand Nagar, Krishnankoil-626 126, Tamil Nadu, India.}
\title{\large \textbf{On $(k, d)$ Hooked Skolem Graceful Graphs}}
\date{}

\begin{document}
\maketitle
\begin{abstract}
A graph $(p, q)$ graph $G = (V, E)$ is said to be $(k, d)$-hooked Skolem graceful if there
exists a bijection $f:V (G)\rightarrow \{1, 2, \dots, p-1, p+1\}$ such that the induced
edge labeling $g_f : E \rightarrow \{k, k+d, \dots, k+(n-1)d \}$ defined by $g_f (uv) =
|f(u) - f(v)|$ $\forall uv \in E$ is also bijective, where $k$ and $d$ are positive integers.
Such a labeling $f$ is called $(k, d)$-hooked Skolem graceful labeling of $G.$
Note that when $k = d = 1$, this notion coincides with that of Hooked Skolem (HS)
graceful labeling of the graph G.
In this paper, we present some preliminary results on $(k, d)$-hooked Skolem graceful graphs and prove that $nK_2$ is $(2, 1)$-hooked Skolem graceful if and only if $n \equiv 1~\mbox{or}~2(\bmod~ 4)$.

\end{abstract}
\textbf{2010 MATHEMATICS SUBJECT CLASSIFICATION}:\\
\textit{\textbf{05C 78}.\\}


\section{\bf Introduction}
By a graph $G = (V, E)$, we mean a finite undirected graph without loops
or multiple edges. The order $|V|$ and the size $|E|$ of $G$ are denoted by $p$
and $q$ respectively. For graph theoretic terminology and notations we refer to
West \cite{wes}.

\bigskip

While studying the structure of Steiner triple systems, Skolem \cite{sko} consid-
ered the following problem: {\it{Is it possible to distribute the numbers $1, 2, \dots, 2p$ into $n$ pairs $(a_i, b_i)$ such
that we have $b_i- a_i = i$ for $i = 1, 2, \dots, p$?}}
In the sequel, a set of pairs of this kind is called $1, +1$ system because the difference $b_i-a_i$ begins with $1$ and increases by $1$ when $i$ increases by $1$. Skolem \cite{sko} proved that a $1, +1$ system exists if and only if $p \equiv 0 ~ \mbox{or} ~ 1(\bmod~4)$. A $1, +1$ system is also known as Skolem sequence, which is defined as follows.

\bigskip

\begin{dfn}
Let $\langle C_i \rangle$ be a sequence of $2p$ terms, where $1 \leq C_i \leq p$. If each number $i$ occurs exactly twice in the sequence and $|j_2 - j_1| = i$ if $i = C_{j_1} = C_{j_2}$ then $\langle C_i \rangle$ is called a {\it Skolem sequence.}
\end{dfn}

\bigskip

This concept was used by Lee and Shee \cite{lee} to introduce the notion of Skolem gracefulness of graphs.

\bigskip

\begin{dfn}
{\it A Skolem graceful labeling} of a $(p, q)$ graph $G = (V, E)$ is a bijection $f : V \rightarrow \{1, 2, \dots, p\}$ such that the induced labeling $g_f : E \rightarrow \{1, 2, \dots, q\}$ defined by $g_f (uv) = |f(u) - f(v)|$ is also a bijection. If such a labeling exists then the graph $G$ is called a {\it Skolem graceful graph}. 
\end{dfn}

\bigskip

If a graph $G$ with $p$ vertices and $q$ edges, is graceful then $q \geq p - 1$, while if it is Skolem graceful, then $q \leq p - 1$.
Thus, as noted in \cite{lee}, Skolem graceful labelings nearly complement graceful labelings, and a graph with $q = p - 1$ is graceful if and only if it is Skolem graceful.

\bigskip

O'Keefe \cite{kee} extended the methods of Skolem sequences for $k \equiv 2 ~ \mbox{or} ~ 3(\bmod~4)$ by showing that the numbers $1, 2, \dots, 2k-1, 2k+1$ can be distributed into $k$ disjoint pairs $(a_i, b_i)$ such that $b_i = a_i + i$ for $i = 1, 2, \dots, k$.

\bigskip

Motivated by this, Shalaby \cite{sha} defined the notion of hooked Skolem sequences as follows.

\bigskip

\begin{dfn}
{\it A hooked Skolem sequence} (HS) of order $k$ is a sequence $(c_1, c_2, \dots, c_{2k+1})$ of $2k + 1$ integers satisfying the following conditions:
\begin{enumerate}
\item For every $r \in \{1, 2, \dots, k\}$ there exist exactly two elements $c_i$ and $c_j$ such
that $c_i = c_j = r$.
\item If $c_i = c_j = r$ with $i < j$, then $j - i = r$.
\item  $c_{2n} = 0$.
\end{enumerate}
\end{dfn}

\bigskip

In \cite{sim} a {\it hooked sequence} is defined as a sequence $\{d, d+1, \dots, d+m-1\}$ for which there is a partition of the set $\{1, 2, \dots, 2m-1, 2m+1\}$ into $m$ pairs $(a_i, b_i)$ such that the $m$ numbers $b_i-a_i$, $1 \leq i \leq m$ are all of the integers $d, d+1, \dots, d+m-1$. Where $a_i$ and $b_i$ are interpreted as the two positions in the sequence where $b_i-a_i$ appears. For example $48574365387*6$ and $64758463573*8$ are hooked sequences where $d=3$ and $m=6$.  

\bigskip

In \cite{TS2} a hooked Skolem graceful graph is defined as follows: A $(p, q)$ graph $G=(V, E)$ is said to be {\it hooked Skolem graceful} if there exists a bijection $f:V(G) \rightarrow \{1, 2, \dots, p-1, p+1\}$ such that the induced edge labeling $g_f:E \rightarrow \{1, 2, \dots, q\}$ defined by $g_f(uv)~=~|f(u)-f(v)|,~\forall uv \in E$ is also bijective. Such a labeling $f$ is called {\it hooked Skolem graceful labeling} of $G$.

\bigskip

In this paper, we introduce the notion of $(k, d)$-hooked Skolem graceful graph as follows:

\bigskip

\begin{dfn}
A $(p, q)$ graph $G=(V, E)$ is said to be {\it $(k, d)$-hooked Skolem graceful} if there exists a bijection $f:V(G) \rightarrow \{1, 2, \dots, p-1, p+1\}$ such that the induced edge labeling $g_f:E \rightarrow \{k, k + d, k + 2d, \dots, k + (q-1)d\}$ defined by $g_f(uv)~=~|f(u)-f(v)|~\forall uv \in E$ is also bijective, where $k$ and $d$ are positive integers. Such a labeling $f$ is called {\it $(k,d)$-hooked Skolem graceful labeling} of $G$.
\end{dfn}

\section{\bf Main Results}
It follows from the definition that if $G$ is $(k, d)$-hooked Skolem graceful, then $q \leq p-1$. For any two disjoint subsets $A$ and $B$ of $V$, we denote by $m(A, B)$ the number of edges of $G$ with one end in $A$ and the other end in $B$.\\

\begin{theorem}
Let $k$ and $d$ be two positive integers which are not simultaneously even. If $G$ is $(k, d)$-hooked Skolem graceful, then $V(G)$ can be partitioned into two subsets $V_o$ and $V_e$ satisfying the following conditions.
\begin{enumerate}
\item $m(V_o,V_e)=\lfloor\frac{q+1}{2}\rfloor$ if $k$ and $d$ are both odd.
\item $m(V_o,V_e)=\lfloor\frac{q}{2}\rfloor$ if $k$ is even and $d$ is odd.
\item $m(V_o,V_e)=q$ if $k$ is odd and $d$ is even.
\end{enumerate}
\end{theorem}

\begin{proof}
Let $f$ be a $(k,d)$-hooked Skolem graceful labeling of $G$. Let $V_o=\{u \in V(G):f(u) \ \mbox{is odd}\}$ and $V_e=V(G)-V_o$. then $g_f(G)$ is odd if and only if the corresponding edge joins a vertex of $V_o$ and a vertex of $V_e$ and hence the result follows.
\end{proof}

\bigskip

In the following theorem we investigate the existence of $(k, d)$-hooked Skolem graceful labeling for $nK_2$.

\bigskip

\begin{thm}\label{th}
If $nK_2$ is $(k,d)$-hooked Skolem graceful, then one of the following holds.
\begin{enumerate}
\item $n \equiv 1(\bmod~4)$, then $k$ is even.\label{a}
\item $n \equiv 2(\bmod~4)$, then $d$ is odd.\label{b}
\item $n \equiv 3(\bmod~4)$, then both $k$ and $d$ are even or they are odd.\label{c}
\end{enumerate}
\end{thm}

\begin{proof}
Let $f$ be a $(k, d)$-hooked Skolem graceful labeling of $nK_2$. Let $e_i = u_iv_i$ be
the components of $nK_2$ and let $f(u_i) = a_i$, $f(v_i) = b_i$ and $b_i > a_i$, $1 \leq i \leq n$.
Since the set of vertex labels is $\{1, 2, \dots, 2n-1, 2n+1\}$ and the set of edge labels is
$\{k, k + d, \dots, k + (n - 1)d\}$, we have
\begin{align}\label{1}
\sum\limits_{i=1}^n (b_i - a_i)  = k +(k+d)+ \dots + k+(n-1)d
\end{align}
\begin{align}\label{2}
\begin{split}
\sum\limits_{i=1}^n b_i + \sum\limits_{i=1}^n a_i  & = 1 + 2 + \dots + (2n-1) + (2n+1) \\
                                    & = n(2n-1) + (2n+1)
\end{split}
\end{align}
On adding (\ref{1}) and (\ref{2}) we have,
\begin{align*}
\sum\limits_{i=1}^n b_i = nk+2n^2+n+1+\frac{n(n-1)}{2}d
\end{align*}

If $n \equiv 1(\bmod~4)$, then $2n^2+n+1+\frac{n(n-1)}{2}$ is even and hence $k$ is even. Hence condition {\ref{a}} holds. By a similar argument conditions {\ref{b}} and {\ref{c}} can be proved.
\end{proof}

\bigskip

\begin{cor}
The necessary conditions for the sequence $\{d, d+1, d+2, \dots, d+m-1\}$ to be hooked are:
\begin{enumerate}
\item $m(m+1-2d)+2 \geq 0$ and
\item $m \equiv 2~ \mbox{or}~ 3(\bmod~4)$ for $d$ odd and $m \equiv 2~ \mbox{or}~ 1(\bmod~4)$ for $d$ even.
\end{enumerate}
\end{cor}

\bigskip

The following theorem gives the necessary and sufficient condition for $nK_2$ to be $(2,1)$-hooked Skolem graceful.

\bigskip

\begin{theorem}
The graph $nK_2$ is $(2, 1)$-hooked Skolem graceful if and only if $n \equiv 1 ~ \mbox{or} ~ 2(\bmod~4)$.
\end{theorem}

\begin{proof}
If $nK_2$ is $(2, 1)$-hooked Skolem graceful, then it follows from Theorem \ref{th} that $n \equiv 1 ~ \mbox{or}~ 2(\bmod~4)$. \\

\ni Conversely, let $n \equiv 1 ~ \mbox{or} ~ 2(\bmod~4)$. Let $e_i=a_ib_i$ be the edges of $nK_2$ with $b_i > a_i$, $1 \leq i \leq n.$\\

\ni{\bf{Case 1:}}\quad $n \equiv 2(\bmod~4)$.

Let $n = 4r-2$, where $r$ is a positive integer. For $r=1,~ 2$ and $3$, the $(2,1)$-hooked Skolem graceful labeling of $2K_2$, $6K_2$ and $10K_2$ are given in Figure {\ref{fig4.1}}.

\begin{figure}[h!]
\begin{center}
\scalebox{1} 
{
\begin{pspicture}(0,-5.120625)(7.53,5.120625)
\psdots[dotsize=0.2](3.34,4.5771875)
\psdots[dotsize=0.2](3.34,3.1771874)
\psdots[dotsize=0.2](4.16,4.5771875)
\psdots[dotsize=0.2](4.16,3.1771874)
\psdots[dotsize=0.2](1.72,1.0171875)
\psdots[dotsize=0.2](1.72,-0.3828125)
\psdots[dotsize=0.2](2.54,1.0171875)
\psdots[dotsize=0.2](2.54,-0.3828125)
\psdots[dotsize=0.2](3.34,1.0171875)
\psdots[dotsize=0.2](3.34,-0.3828125)
\psdots[dotsize=0.2](4.16,1.0171875)
\psdots[dotsize=0.2](4.16,-0.3828125)
\psdots[dotsize=0.2](4.94,1.0371875)
\psdots[dotsize=0.2](4.94,-0.3628125)
\psdots[dotsize=0.2](5.76,1.0371875)
\psdots[dotsize=0.2](5.76,-0.3628125)
\psdots[dotsize=0.2](1.72,-2.5828125)
\psdots[dotsize=0.2](1.72,-3.9828124)
\psdots[dotsize=0.2](2.54,-2.5828125)
\psdots[dotsize=0.2](2.54,-3.9828124)
\psdots[dotsize=0.2](3.34,-2.5828125)
\psdots[dotsize=0.2](3.34,-3.9828124)
\psdots[dotsize=0.2](4.16,-2.5828125)
\psdots[dotsize=0.2](4.16,-3.9828124)
\psdots[dotsize=0.2](4.94,-2.5628126)
\psdots[dotsize=0.2](4.94,-3.9628124)
\psdots[dotsize=0.2](5.76,-2.5628126)
\psdots[dotsize=0.2](5.76,-3.9628124)
\psdots[dotsize=0.2](0.1,-2.5628126)
\psdots[dotsize=0.2](0.1,-3.9628124)
\psdots[dotsize=0.2](0.92,-2.5628126)
\psdots[dotsize=0.2](0.92,-3.9628124)
\psdots[dotsize=0.2](6.52,-2.5428126)
\psdots[dotsize=0.2](6.52,-3.9428124)
\psdots[dotsize=0.2](7.34,-2.5428126)
\psdots[dotsize=0.2](7.34,-3.9428124)
\usefont{T1}{ptm}{m}{n}
\rput(3.446875,4.9471874){1}
\usefont{T1}{ptm}{m}{n}
\rput(4.2185936,4.9471874){2}
\usefont{T1}{ptm}{m}{n}
\rput(3.3476562,2.8671875){3}
\usefont{T1}{ptm}{m}{n}
\rput(4.1295314,2.8471875){5}
\usefont{T1}{ptm}{m}{n}
\rput(1.726875,1.3471875){1}
\usefont{T1}{ptm}{m}{n}
\rput(2.5385938,1.3271875){2}
\usefont{T1}{ptm}{m}{n}
\rput(3.3476562,1.3271875){3}
\usefont{T1}{ptm}{m}{n}
\rput(4.1809373,1.3271875){4}
\usefont{T1}{ptm}{m}{n}
\rput(4.969531,1.3471875){5}
\usefont{T1}{ptm}{m}{n}
\rput(5.723906,1.3071876){13}
\usefont{T1}{ptm}{m}{n}
\rput(1.70875,-0.6528125){8}
\usefont{T1}{ptm}{m}{n}
\rput(2.5153124,-0.7128125){7}
\usefont{T1}{ptm}{m}{n}
\rput(3.3353126,-0.7128125){6}
\usefont{T1}{ptm}{m}{n}
\rput(4.1525,-0.6928125){10}
\usefont{T1}{ptm}{m}{n}
\rput(4.9751563,-0.6928125){9}
\usefont{T1}{ptm}{m}{n}
\rput(5.756875,-0.7128125){11}
\usefont{T1}{ptm}{m}{n}
\rput(0.126875,-2.2728126){1}
\usefont{T1}{ptm}{m}{n}
\rput(0.9585937,-2.2528124){2}
\usefont{T1}{ptm}{m}{n}
\rput(1.7409375,-2.2928126){4}
\usefont{T1}{ptm}{m}{n}
\rput(2.5295312,-2.2928126){5}
\usefont{T1}{ptm}{m}{n}
\rput(3.3153124,-2.2728126){7}
\usefont{T1}{ptm}{m}{n}
\rput(4.12875,-2.2728126){8}
\usefont{T1}{ptm}{m}{n}
\rput(4.9125,-2.3128126){10}
\usefont{T1}{ptm}{m}{n}
\rput(5.696875,-2.3128126){11}
\usefont{T1}{ptm}{m}{n}
\rput(6.5714064,-2.2728126){12}
\usefont{T1}{ptm}{m}{n}
\rput(7.3039064,-2.2928126){13}
\usefont{T1}{ptm}{m}{n}
\rput(0.08765625,-4.3128123){3}
\usefont{T1}{ptm}{m}{n}
\rput(0.8753125,-4.3528123){6}
\usefont{T1}{ptm}{m}{n}
\rput(1.6551563,-4.3728123){9}
\usefont{T1}{ptm}{m}{n}
\rput(2.5270312,-4.2928123){15}
\usefont{T1}{ptm}{m}{n}
\rput(3.2726562,-4.3328123){14}
\usefont{T1}{ptm}{m}{n}
\rput(4.0909376,-4.2928123){17}
\usefont{T1}{ptm}{m}{n}
\rput(4.9540625,-4.2928123){21}
\usefont{T1}{ptm}{m}{n}
\rput(5.75125,-4.3328123){19}
\usefont{T1}{ptm}{m}{n}
\rput(6.5285935,-4.3128123){18}
\usefont{T1}{ptm}{m}{n}
\rput(7.2920313,-4.3328123){16}
\psline[linewidth=0.04cm](3.34,4.4971876)(3.34,3.2771876)
\psline[linewidth=0.04cm](4.14,4.4971876)(4.14,3.2771876)
\psline[linewidth=0.04cm](1.72,0.9371875)(1.72,-0.2828125)
\psline[linewidth=0.04cm](2.54,0.9771875)(2.56,-0.3428125)
\psline[linewidth=0.04cm](3.34,0.9771875)(3.34,-0.4028125)
\psline[linewidth=0.04cm](4.16,0.9571875)(4.16,-0.3628125)
\psline[linewidth=0.04cm](4.92,0.9971875)(4.94,-0.3228125)
\psline[linewidth=0.04cm](5.74,0.9971875)(5.74,-0.2828125)
\psline[linewidth=0.04cm](0.1,-2.6028125)(0.1,-3.9228125)
\psline[linewidth=0.04cm](0.9,-2.6228125)(0.9,-3.9028125)
\psline[linewidth=0.04cm](1.7,-2.6428125)(1.7,-3.9428124)
\psline[linewidth=0.04cm](2.52,-2.6228125)(2.54,-3.9428124)
\psline[linewidth=0.04cm](3.32,-2.5828125)(3.32,-3.9428124)
\psline[linewidth=0.04cm](4.14,-2.5828125)(4.14,-3.9228125)
\psline[linewidth=0.04cm](4.94,-2.6228125)(4.92,-3.9228125)
\psline[linewidth=0.04cm](5.72,-2.5828125)(5.74,-3.8828125)
\psline[linewidth=0.04cm](6.5,-2.5828125)(6.52,-3.8828125)
\psline[linewidth=0.04cm](7.34,-2.5828125)(7.34,-3.8828125)
\usefont{T1}{ptm}{m}{n}
\rput(3.7414062,2.3471875){$2K_2$}
\usefont{T1}{ptm}{m}{n}
\rput(3.7414062,-1.2728125){$6K_2$}
\usefont{T1}{ptm}{m}{n}
\rput(3.7114062,-4.8928127){$10K_2$}
\end{pspicture}
}
\caption{$(2,1)$-hooked Skolem graceful labeling of $2K_2$, $6K_2$ and $10K_2$}{\label{fig4.1}}
\end{center}
\end{figure}
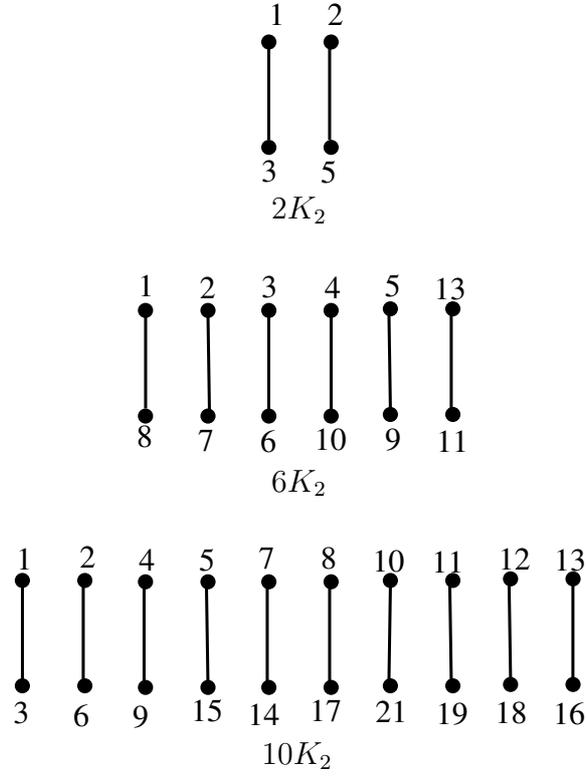

\ni For $r \geq 4$, we define the vertex labeling $f$ as follows:

\begin{eqnarray*}
f(a_i)=
\begin{cases}
i, & \text {for $i =1,2;$}\\
i+1, & \text {for $3 \leq i \leq 2r-2;$}\\
\frac{n+4}{2}, & \text {for $i= 2r-1;$}\\
\frac{3n+2}{4}, & \text {for $i= 2r;$}\\
\frac{n-4}{2}+i, & \text {for $2r+1 \leq i \leq n.$}
\end{cases}
\end{eqnarray*}

\begin{eqnarray*}
f(b_i)=
\begin{cases}
3, & \text {for $i=1;$}\\
\frac{n+2}{2}, & \text {for $i = 2;$}\\
n+2-i, & \text {for $3 \leq i \leq r;$}\\
n+1-i, & \text {for $r+1 \leq i \leq 2r-3;$}\\
\frac{3n}{2}, & \text {for $i = 2r-2;$}\\
\frac{3n-2}{2}, & \text {for $i = 2r-1;$}\\
\frac{7n-2}{4}, & \text {for $i = 2r;$}\\
2n+1, & \text {for $i = 2r+1;$}\\
\frac{5n+4}{2}-i, & \text {for $2r+2 \leq i \leq 3r;$}\\
\frac{5n+2}{2}-i, & \text {for $3r+1 \leq i \leq n.$}
\end{cases}
\end{eqnarray*}\\

\noindent
{\bf{Case 2:}}\quad $n \equiv 1(\bmod~4).$

Let $n = 4r - 3$, where $r$ is a positive integer.
For $r=1$ and $2$ the $(2,1)$-hooked Skolem graceful labelings of $K_2$ and $5K_2$ are given in Figure {\ref{fig4.2}}.

\begin{figure}[h!]
\begin{center}
\scalebox{1} 
{
\begin{pspicture}(0,-3.830625)(3.3990624,3.830625)
\psdots[dotsize=0.2](1.734375,3.3871875)
\psdots[dotsize=0.2](1.734375,1.9671875)
\psline[linewidth=0.04cm](1.734375,3.3271875)(1.734375,2.0671875)
\psdots[dotsize=0.2](0.114375,-1.0728126)
\psdots[dotsize=0.2](0.114375,-2.4928124)
\psline[linewidth=0.04cm](0.114375,-1.1328125)(0.114375,-2.3928125)
\psdots[dotsize=0.2](0.914375,-1.0528125)
\psdots[dotsize=0.2](0.914375,-2.4728124)
\psline[linewidth=0.04cm](0.914375,-1.1128125)(0.914375,-2.3728125)
\psdots[dotsize=0.2](1.714375,-1.0528125)
\psdots[dotsize=0.2](1.714375,-2.4728124)
\psline[linewidth=0.04cm](1.714375,-1.1128125)(1.714375,-2.3728125)
\psdots[dotsize=0.2](2.454375,-1.0528125)
\psdots[dotsize=0.2](2.454375,-2.4728124)
\psline[linewidth=0.04cm](2.454375,-1.1128125)(2.454375,-2.3728125)
\psdots[dotsize=0.2](3.274375,-1.0328125)
\psdots[dotsize=0.2](3.274375,-2.4528124)
\psline[linewidth=0.04cm](3.274375,-1.0928125)(3.274375,-2.3528125)
\usefont{T1}{ptm}{m}{n}
\rput(1.70125,3.6571875){1}
\usefont{T1}{ptm}{m}{n}
\rput(1.7220312,1.6371875){3}
\usefont{T1}{ptm}{m}{n}
\rput(0.10125,-0.7428125){1}
\usefont{T1}{ptm}{m}{n}
\rput(0.93296874,-0.7428125){2}
\usefont{T1}{ptm}{m}{n}
\rput(1.7220312,-0.7628125){3}
\usefont{T1}{ptm}{m}{n}
\rput(2.4439063,-0.7828125){5}
\usefont{T1}{ptm}{m}{n}
\rput(3.2496874,-0.7428125){7}
\usefont{T1}{ptm}{m}{n}
\rput(0.0753125,-2.8228126){4}
\usefont{T1}{ptm}{m}{n}
\rput(0.8896875,-2.8228126){6}
\usefont{T1}{ptm}{m}{n}
\rput(1.663125,-2.8228126){8}
\usefont{T1}{ptm}{m}{n}
\rput(2.43125,-2.8028126){11}
\usefont{T1}{ptm}{m}{n}
\rput(3.2695312,-2.7828126){9}
\usefont{T1}{ptm}{m}{n}
\rput(1.6457813,0.8171875){$K_2$}
\usefont{T1}{ptm}{m}{n}
\rput(1.6557813,-3.6028125){$5K_2$}
\end{pspicture}
}
\caption{$(2,1)$-hooked Skolem graceful labeling of $K_2$ and $5K_2$}{\label{fig4.2}}
\end{center}
\end{figure}

\ni For $r \geq 3$, we define the vertex labeling $f$ as follows:
\begin{align*}
f(a_i)&=
\begin{cases}
i, & \text {for $1 \leq i \leq 2r-1;$}\\
\frac{n-3}{2}+i, & \text {for $2r \leq i \leq 3r-2;$}\\
\frac{n-1}{2}+i, & \text {for $3r-1 \leq i \leq n;$}
\end{cases}
\end{align*}
\begin{align*}
f(b_i)&=
\begin{cases}
n-i, & \text {for $1 \leq i \leq r-1;$}\\
n-1+i, & \text {for $i = r, n;$}\\
n+1-i, & \text {for $r+1 \leq i \leq 2r-2;$}\\
\frac{3n+1}{2}, & \text {for $i = 2r-1;$}\\
2n+1, & \text {for $i = 2r;$}\\
\frac{5n+1}{2}-i, & \text {for $2r+1 \leq i \leq n-1.$}
\end{cases}
\end{align*}

\noindent
In each case, it can be easily verified that the induced edge function $g_f$ defined by $g_f(e_i) = |b_i - a_i|$ has the required properties to qualify $f$ to be a $(2, 1)$-hooked Skolem graceful labeling of $nK_2$ and the cases exhaust all the possibilities. This completes the proof.
\end{proof}

\section{ Conclusion and Scope}
In this paper, we have introduced the notion of $(k,d)$-hooked Skolem graceful graphs and observe that $k=d=1$ coincides with the notion of hooked Skolem graceful labeling of a graph $G$. We have given some necessary or sufficient conditions for a graph $G$ to be $(k,d)$-hooked Skolem graceful. We have proved that $nK_2$ is $(2,1)$-hooked Skolem graceful if and only if $n \equiv 1 ~ \mbox{or} ~ 2(\bmod~4)$. Determining the value of $n$ for which $nK_2$ is $(k,d)$-hooked Skolem graceful for given values of $k$ and $d$ is an open problem.

\end{document}